\documentclass[12pt]{amsart}
\usepackage{amsmath,amssymb}
\usepackage{graphicx,subfigure}
\usepackage{mathtools,slashed}
\newtheorem{theorem}{Theorem}[section]
\newtheorem{proposition}[theorem]{Proposition}
\newtheorem{corollary}[theorem]{Corollary}

\newtheorem{lemma}[theorem]{Lemma}

\theoremstyle{definition}
\newtheorem{definition}[theorem]{Definition}

\theoremstyle{remark}
\newtheorem{remark}[theorem]{Remark}

\numberwithin{equation}{section}

\def\CC{\mathbb{C}}
\def\NN{\mathbb{N}}

\def\RR{\mathbb{R}}

\def\ZZ{\mathbb{Z}}
\def\RP{\mathbb{RP}}

\def\TT{\mathbb{T}}
\def\MM{{\mathbb{M}}}

\renewcommand\SS{\mathbb{S}}

\newcommand\minus\backslash

\newcommand\lan\langle
\newcommand\ran\rangle

\DeclareMathOperator\Div{div}

\renewcommand\leq\leqslant
\renewcommand\geq\geqslant
\newlength{\intwidth}

\newcommand{\vect}{\mathfrak{X}}

\begin{document}

\title[Chaos in the Euler equation]{Chaos in the incompressible Euler equation on manifolds of high dimension}

\author{Francisco Torres de Lizaur}
\address{Department of Mathematics, University of Toronto, Toronto, ON M5S 2E4, Canada}
\email{ftlizaur@math.toronto.edu}

\begin{abstract}

We construct finite dimensional families of non-steady solutions to the Euler equations, existing for all time, and exhibiting all kinds of qualitative dynamics in the phase space, for example: strange attractors and chaos, invariant manifolds of arbitrary topology, and quasiperiodic invariant tori of any dimension. 

The main theorem of the paper, from which these families of solutions are obtained, states that for any given vector field $X$ on a closed manifold $N$, there is a Riemannian manifold $M$ on which the following holds: $N$ is diffeomorphic to a finite dimensional manifold in the space of divergence-free vector fields on $M$ that is invariant under the Euler evolution, and on which the Euler equation reduces to a finite dimensional ODE that is given by an arbitrarily small perturbation of the vector field $X$ on $N$.  

\end{abstract}

\maketitle

\section{Introduction}
The motion of an incompressible inviscid fluid in a Riemannian manifold $(M, g)$ is described by a time-dependent vector field $u_t$ satisfying the \emph{Euler equations}

\begin{equation}\label{euler}
\partial_t u_t+\nabla^{LC}_{u_t} u_t= -\nabla p ,\,\,\,\,\, \Div u_t=0 \,.
\end{equation}

Here $\Div$, $\nabla$, and $\nabla^{LC}$ denote the divergence, gradient and covariant derivative on $M$, defined with the metric $g$; and $p$ is a (time-dependent) function called the pressure, also an unknown in the equations.

These equations define a dynamical system of an infinite number of degrees of freedom: they can be interpreted as a first order ODE (which we will call the Euler system) in the infinite-dimensional linear space $\vect^{m}(M)$ of $m$-times differentiable divergence-free vector fields on the domain $M$. 

We will denote by $\Phi_t(v):=u_t$ the flow of this ODE starting at the field $v \in \vect^{m}$. Small time existence and uniqueness hold for $M$ closed (compact without boundary) and non-integer $m>1$ \cite{EB}; moreover, the solution $u_t$ is $C^{1}$ on the variable $t$.

Global in time existence of solutions, however, is only known in 2 dimensions \cite{Wo}; in higher dimensions, whether or not there are initial conditions for which the flow $\Phi_t$ ``blows-up'' in finite time (i.e $||\Phi_t(v)||_{C^{m}(M)} \rightarrow \infty$ as $t \rightarrow T < \infty$)  is a well-known open problem.

As with dynamical systems of a finite number of degrees of freedom, besides existence and uniqueness of solutions we would like to understand the qualitative properties of the flow $\Phi_t$, when it exists. Examples of these properties are the number of equilibrium points and their stability, the periodicity and almost periodicity of trajectories, and the geometry and dynamics of more complex invariant subsets  (other qualitative properties of the Euler flow $\Phi_t$ for which some results are known are mixing \cite{KKPS1, KKPS2} and wandering of solutions \cite{Na, Sn}, but our results will imply nothing about these). 

So far, the only invariant sets whose existence has been established are the simplest: stationary solutions (i.e, zeros), heteroclinic and homoclinic trajectories between zeroes, periodic orbits, and quasiperiodic invariant tori \cite{CF, Taoquadratic}; in these invariant sets, of course, the flow $\Phi_t$ is well defined for all times.  

Observe that, in the finite-dimensional invariant manifolds of $\Phi_t$, the evolution of the fluid velocity in time can be completely described by the evolution of a finite set of parameters; in other words, the Euler equation reduces to a finite dimensional first order ODE on the invariant manifold. The previous paragraph implies that the Euler equation is only known to reduce to the simplest finite dimensional ODEs, conjugate to linear periodic or quasiperiodic flows on tori. 

The goal of this paper is to show that, in fact, almost any finite dimensional smooth dynamics can be found in the phase space of the Euler system, in some Riemannian manifold. More precisely, our main result states:

\begin{theorem}\label{main}
Let $N$ be any closed (compact without boundary) manifold. Given any $C^{\infty}$ vector field $X$ on $N$, and two positive numbers $\epsilon$ and $m$, there is 
\begin{enumerate}

\item a $C^{\infty}$ vector field $Y$ in $N$ satisfying 
\[
||X-Y||_{C^{m}(N)} \leq \epsilon \,,
\]

\item a compact Riemannian manifold $M$, together with a finite-dimensional linear subspace $E \subset \vect^{\infty}(M)$

\item a $C^{\infty}$ embedding $\Theta: N \rightarrow E$

\end{enumerate}

such that for any point $p \in N$, we have that the time-dependent divergence-free field $u_t=\Theta(\phi^t_{Y}(p))$ (here $\phi^t_{Y}$ stands for the flow of $Y$) is the solution to the Euler equation \eqref{euler} on $M$ with initial velocity $u_0=\Theta(p)$.

\end{theorem}

\begin{remark}
The embedding $\Theta$ in Theorem  \ref{main} is $C^{\infty}$ for any $C^{k}$ ($k\in [0, \infty]$) topology in the space $\vect^{\infty}(M)$ of smooth divergence-free vector fields on $M$. In what follows, by smooth we will always mean $C^{\infty}$.
\end{remark}

We will deduce Theorem \ref{main} from Theorem \ref{main2} below, whose proof will constitute the core of the paper:

\begin{theorem}\label{main2}
Let $\MM$ denote the $n$-sphere $\SS^{n}$ or the $n$-dimensional torus $\TT^n=(\RR/\ZZ)^n$, for any $n\geq 1$. Let $X$ be any vector field on $\MM$ that can be expressed as a trigonometric polynomial (in the case of $\TT^n)$ or as a polynomial vector field in $\RR^{n+1}$ tangent to $\SS^n$. Then, there is:
\begin{enumerate}

\item a compact Riemannian manifold $M$, together with a finite-dimensional linear subspace $E \subset \vect^{\infty}(M)$

\item a smooth embedding $\Phi: \MM \rightarrow E$

\end{enumerate}

such that for any point $x \in \MM$, the time-dependent divergence-free field $u_t=\Phi(\phi^t_{X}(x))$ is the solution to the Euler equation \eqref{euler} on $M$ with initial condition $u_0=\Phi(x)$.

\end{theorem}

\begin{remark}
The dimension of the manifold $M$ in Theorem \ref{main2} can be computed exactly and depends only on $n$ and on the degree of the vector field $X$; we refer to Subsection \ref{dimension2} for more details.  In the case of Theorem \ref{main}, the dependence of the dimension of $M$ is more complex, we discuss it and find an upper bound in Subsection \ref{dimension}.
\end{remark}

Let us present rough sketches of the proofs. 

To prove Theorem \ref{main2}, the key is to construct a homogeneous quadratic ODE in $\RR^{d}$, for $d$ big enough, having an invariant manifold diffeomorphic to $\MM$ where the flow of the ODE is conjugate to the flow of the vector field $X$. This ODE, moreover, will be shown to preserve the standard euclidean inner product (in other words, its trajectories are always tangent to spheres), so we can embed it into the Euler equations on the manifold $SO(d)\times \TT^{d}$ with a certain metric, applying a Theorem of Tao \cite{Taoquadratic}.

To prove Theorem  \ref{main}, we embed the manifold $N$ in a sphere $\SS^{n}$ of high enough dimension, then extend the vector field $X$ on $N$ to $ \SS^n$ in a way that guarantees structural stability of $N$ under sufficiently small perturbations of the extension of $X$. This ensured, we approximate the extension of $X$ by a polynomial vector field, and apply Theorem \ref{main2}.

\begin{remark}
The results above can be interpreted as a universality result for the (time-dependent) Euler equation. Let us remark that this is different from the recent universality result obtained for steady Euler in \cite{CMDP}. There it is shown, using techniques from contact geometry, that any vector field on a compact manifold can be embedded, without perturbation, into a stationary solution of the Euler equations on a Riemannian manifold of higher dimension.
\end{remark}

\subsection{Applications of Theorems \ref{main} and \ref{main2}}

The results above can be used to construct solutions to the Euler equations with dynamical structures that were previously unknown, or at best conjectural. In particular, there are finite dimensional invariant submanifolds in the phase space of the Euler equation where the evolution is chaotic:

\begin{theorem}\label{chaos}

There is a Riemannian manifold $(M, g)$ of dimension $21$, for which the Euler dynamical system is chaotic. More precisely, there is a 3-dimensional torus $\Sigma$ inside $\vect^{\infty}(M, g)$ such that
\begin{enumerate}
\item the solutions to the Euler equation with initial condition in $\Sigma$ exists for all time, and remain in $\Sigma$ (i.e, $\Sigma$ is an invariant torus of the Euler dynamical system)
\item the Euler dynamical system on $\Sigma$ has a compact invariant set where the dynamic is chaotic (presence of transverse homoclinic intersections, horseshoes, and positive topological entropy); as well as other regions foliated by 2-dimensional invariant tori.
\end{enumerate}

\end{theorem}

We will prove this Theorem in Section \ref{s4}, as an application of Theorem  \ref{main2}.

We can also use Theorem \ref{main} to prove that some submanifolds of phase space are filled with strange attractors of hyperbolic type (for example, Smale-Williams solenoids). 
Indeed, let $X$ be a vector field on some manifold $N$ having a hyperbolic strange attractor $A$. This means that any sufficiently small $C^m$ perturbation of $X$ also has a hyperbolic strange attractors $A'$ close to $A$, on which the flow is topologically conjugate to the flow of $X$ in $A$. Thus we have

\begin{corollary}[Hyperbolic strange attractor]
On the space of smooth divergence-free vector fields on some Riemannian manifolds $M$, we can find invariant finite-dimensional subsets on which the Euler equation has solutions for all time that converge to an hyperbolic strange attractor. 
\end{corollary}

\begin{remark}
The Lorenz attractor is not hyperbolic, so the previous corollary does not apply to it. However, there is a sense in which we can still find it in the Euler system. We refer to Section \ref{sf}. 
\end{remark}

Another consequence is that we can find finite dimensional families of trajectories that are Anosov, since Anosov flows are stable under small $C^{1}$ perturbations:

\begin{corollary}[Anosov flows]
On the space of smooth divergence-free vector fields on some Riemannian manifolds $M$, we can find finite-dimensional manifolds on which the Euler equation has solutions for all time and the dynamics is Anosov.
\end{corollary}

As a explicit example, we can consider a small perturbation of the geodesic flow on the unit tangent bundle of a genus $g\geq 2$ Riemannian surface with constant negative curvature.

More generally, Theorem \ref{main} implies that any dynamics displayed by a vector field on a closed manifold that is structurally stable under small $C^m$ perturbations is found in the Euler equations. 

Finally, let us point out that, by virtue of Theorem \ref{main2}, certain well-known types of steady state solutions of the Euler and Navier Stokes equations in $\TT^2$, $\TT^3$ and $\SS^3$ can be embedded exactly into the phase space of the Euler equations in a higher-dimensional manifold $M$. That is, the lagrangian trajectories of the fluid particles in these steady solutions correspond exactly to eulerian trajectories of non-steady vector fields obeying the Euler equation on $M$. More specifically, we are referring to Kolmogorov steady flows in $\TT^2$ (steady solutions of the Navier-Stokes equations under forcing) and Beltrami fields in $\TT^3$ and $\SS^3$, as they have finite Fourier (or spherical harmonics) expansions.

The paper is organized as follows. The proof of Theorem \ref{main2} on embedding dynamics in $\SS^n$ and $\TT^n$ is given in Section \ref{s1}, with the key Proposition (Proposition \ref{main}) proven in Section \ref{s11}. Section \ref{s2} proves Theorem \ref{main}. We apply these results to prove Theorem \ref{chaos} on chaotic dynamics in Section \ref{s4}. We conclude with some further constructions and corollaries.

\section{Proof of Theorem \ref{main2}}\label{s1}

The proof is broken down into two steps.

\subsection{Step 1}

In the following Proposition, we will call ``polynomic'' those vector fields on $\SS^n$ and $\TT^n$ that can be written as finite sums of sines and cosines (in the case of $\TT^n$) or as polynomial vector fields in $\RR^{n+1}$ that are tangent to $\SS^n$. We define the degree of the polynomic vector field to be the modulus squared of the highest frequency (in $\TT^n$) or the degree of the polynomial in $\RR^{n+1}$. 

\begin{proposition}\label{aux1}
Let $X$ be a polynomic vector field on $\MM$ (=$\TT^{n}$ or $\SS^{n}$). There is a smooth embedding $\Psi: \MM \rightarrow \RR^{d}$ (with $d$ depending only on $n$ and on the degree of the vector field) and a homogeneous quadratic ODE on $\RR^{d}$
\[
\frac{d y_i}{dt}=B_{ijk} y_j y_k \text{ for $i=1,..,d$}\,,
\]
satisfying $B_{ijk}=-B_{kji}$, such that for any point $x \in \MM$, $y(t)=\Psi(\phi^{t}_{X}(x))$ is the solution of the quadratic ODE with initial condition $y(0)=\Psi(x)$. 

Moreover, if the vector field $X$ is divergence-free in $\MM$ (with respect to the round metric in $\SS^n$, or to the flat metric in $\TT^{n}$), then the corresponding quadratic vector field
\[
V(y)=\sum_{ijk} B_{ijk} y_j y_k \frac{\partial}{\partial y_i}
\] 
is also divergence-free in $\RR^d$ with the standard Euclidean metric.

\end{proposition}

We prove this proposition in Section \ref{s11}. The fact that the vector field in $\RR^{d}$ is divergence-free when $X$ is will not be needed in what follows, but we found it might be of interest.

\subsection{Step 2}

The next and final step consists in embedding the quadratic ODE obtained in step 1 into the phase space of the Euler equation in some manifold $M$. We apply the following special case of Theorem 1.1 in \cite{Taoquadratic}, which we restate here in a slightly different wording, adapted to our setting:

\begin{theorem}[T. Tao \cite{Taoquadratic}]\label{tao}
Let 
\[
\frac{d y_i}{dt}= V_i(y)=\sum_{j, k=1}^{d} \tilde{B}_{ijk} y_j y_k \text{ for $i=1,..,d$}
\]
be a homogeneous quadratic ODE on $\RR^{d}$, with $\tilde{B}_{ijk}=\tilde{B}_{ikj}$ and
\begin{equation}\label{conditiontao}
\sum_{i, j, k=1}^{d} \tilde{B}_{ijk} y_i y_j y_k=0\,.
\end{equation}
Then there is a Riemannian manifold $M$ (that can be explicitly taken to be $SO(d) \times \TT^{d}$) and a linear injective map $T: \RR^{d} \rightarrow \vect^{\infty}(M)$ such that, for any $y \in \RR^d$, the time-dependent vector field $u_{t}:=T(\phi^{t}_{V}(y))$ is a smooth solution to the Euler equation on $M$ with initial condition $u_0=T(y)$.
\end{theorem}

Observe that the coefficients $B_{ijk}$ in Proposition \ref{aux1} do not directly satisfy the hypothesis in Theorem \ref{tao}, because they are not symmetric in the $j, k$ indices. Nevertheless, setting
\[
\tilde{B}_{ijk}:=\frac{1}{2}(B_{ijk}+B_{ikj})
\]
we see that the coefficients $\tilde{B}_{ijk}$ and $B_{ijk}$ define the same ODE. Moreover, by virtue of Proposition \ref{aux1} we have that  $B_{ijk}=-B_{kji}$, so
\[
\sum_{i, j, k=1}^{d} \tilde{B}_{ijk} y_i y_j y_k=0 \,.
\]

Thus we can apply Theorem \ref{tao} to the ODE given by the coefficients $\tilde{B}_{ijk}$. The embedding $\Phi$ in the statement of Theorem \ref{main2} is then given by $\Phi=T \circ \Psi$. 

To conclude the proof of Theorem \ref{main2}, it remains to be shown that $\Phi=T \circ \Psi$ is smooth for any $C^{k}$ topology (with $k\in [0, \infty]$) in the space of smooth divergence-free fields $\vect^{\infty}(M)$.

For finite $k$ this is obvious, because the embedding $\Psi: \MM \rightarrow \RR^{d}$ is smooth and $T$ is a linear injective map. Indeed, let $\{e_{1},...,e_{d}\}$ denote the standard basis of $\RR^{d}$; the images $\{w_{\mu}=T(e_{\mu})\}$ are a basis of the finite dimensional linear subspace $E \subset \vect^{\infty}$ where we embed $\MM$. The $C^k$ norms of the basis elements are bounded, $||w_{\mu}||_{C^{k}(M)}\leq C(k)$, so the smoothness of $\Psi: \MM \rightarrow \RR^{d}$ immediately implies that of $\Phi=T \circ \Psi: \MM\rightarrow E$.

To see that $\Phi$ is smooth for the $C^{\infty}$ topology in $\vect^{\infty}(M)$, we must show that the $C^{k}$-norms of the vector fields $w_{
\mu}=T(e_{\mu})$ are bounded independently of $k$. This requires a more detailed  discussion of the nature of the map $T$ in Tao's Theorem. 

 The map $T$ has the form
\[
T(e_{\mu})=\Pi^{*} U_{\mu}+\sum_{\nu=1}^{d} \Pi^{*} F_{\mu \nu} \frac{\partial}{\partial t_{\nu}}
\]
where $\Pi$ is the projection map $\Pi: SO(d)  \times \TT^{d} \rightarrow SO(d)$, $\{\frac{\partial}{\partial t_{\nu}}\}_{\nu=1}^{d}$ is the standard basis of  $T \TT^{d}$, and $U_{\mu}$ and $F_{\mu \nu}$ are vector fields and functions on $SO(d)$, respectively, that we will describe below: our goal is to show that their $C^{k}$ norms are bounded uniformly in $k$, so that the vector fields $w_{\mu}=T(e_{\mu})$ have the same property.

Consider $SO(d)$ as a submanifold in the linear space Mat$(d)$ of $d\times d$ matrices, so that we can identify the tangent space $T_{Q} SO(d)$ at any orthogonal matrix $Q$ with a linear subspace of Mat$(d)$. As defined in \cite{Taoquadratic}, Section 5, for any given index $\mu=1,...,d$, $U_{\mu}$ is the right invariant vector field on $SO(d)$ whose value at any $Q \in SO(d)$ is given by
\[
U_{\mu}(Q)= S_{\mu} Q
\]
where $S_{\mu}$ is a matrix in $\mathfrak{so}(d)$ that depends only on the coefficients $B_{ijk}$ of our quadratic ODE. As for the functions $F_{\mu \nu}: SO(d)\rightarrow \RR$, they are defined as
\[
F_{\mu \nu}(Q)= e_{\mu} \cdot Q e_{\nu}=Q_{\nu \mu} \,.
\]
To see that the $C^{k}$-norms of the functions and the vector fields are bounded, consider the standard basis $\{E_{\alpha}\}$ of the Lie algebra $\mathfrak{so}(d)$ (consisting of antisymmetric matrices having only two non-zero entries, one equal to $1$ and the other to $-1$). 
The derivative of the function $F_{\mu \nu}$ in the direction given by the tangent vector $Q E_{\alpha_1}$ is given by
\[
\frac{d}{dt}\bigg|_{t=0} F_{\mu \nu}(Q e^{tE_{\alpha_1}})=e_{\mu} \cdot Q E_{\alpha_1} e_{\nu}
\]
Iterating this, we see that the derivatives of order $k$ at the point $Q$ of the function $F_{\mu \nu}$ have the form
\[
e_{\mu} \cdot Q E_{\alpha_k}... E_{\alpha_1} e_{\nu}
\]
The products of the matrices $E_{\alpha}$ have at most one non-zero element in any row or column, equal to $1$ or $-1$. Thus the $k$-derivatives are always equal to some single entry $Q_{\rho \lambda}$ of the matrix $Q$, hence the $C^{k}$ norms of the functions $F_{\mu \nu}$ are bounded independently of $k$. 

As for the $C^{k}$-norms of the right-invariant vector fields $U_{\mu}$, it is enough to check them at the identity. In other words, it is enough to prove that any $k$-iterated commutators
\[
[E_{\alpha_1},[E_{\alpha_2},[....,[E_{\alpha_k}, S_{\mu}]...]]
\]
have norms (for any metric on $\mathfrak{so}(d))$ that can be bounded independently of $k$. This easily follows from the fact that the commutators of the basis elements  $\{E_{\alpha}\}$ are all of the type $[E_{\alpha}, E_{\beta}]=\pm E_{\gamma}$.

\subsection{The dimension of the manifold $M$}\label{dimension2}
As the proof of the Proposition \ref{aux1}, given in Section \ref{s11} below, will make manifest, the parameter $d$ in $SO(d)\times \TT^{d}$ is at most the dimension of the linear space of trigonometric polynomials of degrees up to the degree, say $D$, of $X$ (when $\MM=\TT^n)$; or, in the case of $\SS^n$, at most the dimension of the space of harmonic polynomials of degrees up to one minus the degree of $X$. 

More explicitly, in the case of the torus, this dimension is equal to the number of points in $\ZZ^n$ that lie in the (closed) ball of radius $\sqrt{D}$; while in the case of $\SS^n$, it is equal to
\[
d=\sum_{j=0}^{D}\binom{j+n-1}{n}\frac{2j+n-1}{j+n-1}=\binom{D+n}{n}\frac{D(2Dn+n^2+1)}{n(n+1)(D+n)}
\]

The dimension of $M=SO(d)\times \TT^{d}$ grows then as the square of these numbers.

\section{Proof of Proposition \ref{aux1}}\label{s11}

\subsection{Proof for $\MM=\TT^{n}$}

The field $X$ can be written as a sum of the form
\[
X(x)=\sum_{k \in \ZZ^n, \,|k| \leq \Lambda} a_{k} \sin(2 \pi k \cdot x)+b_{k} \cos(2 \pi k \cdot x) \,\, 
\]
with $a_{k}, b_k \in \RR^n$ and $a_{k}=-a_{-k}$, $b_{k}=b_{-k}$.

Let $d(\Lambda)$ be the number of integer lattice points contained in the $n$-ball of radius $\Lambda$. Consider the map $\Psi: \TT^{n} \rightarrow \RR^{2d(\Lambda)}$ given by
\[
\Psi(x)=\{\sin(2 \pi k \cdot x), \, \cos(2 \pi k \cdot x)\}_{k \in \ZZ^{n}\cap B^{n}(0, \Lambda)}
\]
where we represent the points $(q, p) \in \RR^{2d(\lambda)}$ as $(q, p):=\{q_{k}, p_k\}_{k \in \ZZ^{n}\cap B^{n}(0, \Lambda)}$.

\begin{lemma}\label{emb}
For $\Lambda \geq 1$, the map $\Psi$ is an embedding.

\end{lemma}

\begin{proof}[Proof of Lemma \ref{emb}]
It suffices to prove the claim for $\Lambda=1$. Denote by $B$ the $n$-dimensional ball of radius 1. The frequencies $k \in \ZZ^{n}\cap B$ are the ones of the form
\[
k=\pm(0, 0,...,1,0,...,0)
\]
plus the zero vector, so $d(\Lambda)=2n+1$. 

First we show that $d\Psi$ is injective. Suppose that at some point $x$ there is a vector $v$ in the kernel of the differential :

\[
d_x \Psi (v)= \sum_{k}  2 \pi v \cdot k \cos(2 \pi k \cdot x) \frac{\partial}{\partial q_{k}}-2 \pi v \cdot k \sin(2 \pi k \cdot x)\frac{\partial}{\partial p_{k}}=0
\]
this means that $v \cdot k=0$ for all $k \in \ZZ^n \cap B$. But the frequencies $k$ span the whole $\RR^n$, so we must have $v=0$. 

It remains to be shown that for any two distinct points $x, y \in \TT^n$ we must have $\Psi(x)\neq \Psi(y)$. This is easy to see, for if $\Psi(x)=\Psi(y)$,
\[
\cos(2 \pi k \cdot x)=\cos(2 \pi k \cdot y) \,,\,\,\,\sin(2 \pi k \cdot x)=\sin(2 \pi k \cdot y) 
\]
for integer frequencies $k$ spanning $\RR^n$. This is only possible if $x=y +2 \pi m$ for some $m \in \ZZ^n$, that is, if $x$ and $y$ label the same point in $\TT^n$.
\end{proof}

\begin{remark}\label{harmonics}
If the Fourier expansion of $X$ does not have a constant ($k=0$) component, we can define the map $\Psi$ as 
\[
\Psi(x)=\{\sin(2 \pi k \cdot x), \, \cos(2 \pi k \cdot x)\}_{k \in \ZZ^{n}\cap B^{n}(0, \Lambda)\setminus\{0\}}\,.
\]
It can be readily checked that this does not affect the assertion in Lemma \ref{emb}, nor any further step in the proof. 
\end{remark}

Consider now the vector field $\Psi_{*}(X)$ in $\Psi(\TT^n)$. It has the form

\begin{align*}
d_x \Psi(X)= \sum_{k, k' \in \ZZ^{n}\cap B^{n}(0, \Lambda)} 2\pi \bigg( a_{k'} \cdot k \sin(2 \pi k' \cdot x) \cos(2 \pi k \cdot x)+ \\
+b_{k'} \cdot k \cos(2 \pi k'\cdot x) \cos(2 \pi k \cdot x) \bigg) \frac{\partial}{\partial q_k} \\ -2 \pi \bigg( a_{k'} \cdot k \sin(2 \pi k' \cdot x) \sin(2 \pi k \cdot x)+ \\
+b_{k'} \cdot k \cos(2 \pi k'\cdot x) \sin(2 \pi k \cdot x) \bigg) \frac{\partial}{\partial p_k}
\end{align*} 
Consider as well the following vector field on the whole space $\RR^{2d(\lambda)}$

\[
V(q, p)=2 \pi \sum_{k} \sum_{k'} (a_{k'} \cdot k \,q_{k'} +b_{k'} \cdot k \, p_{k'}) p_k \frac{\partial}{\partial q_k}-(a_{k'} \cdot k \, q_{k'} +b_{k'} \cdot k \, p_{k'}) q_k \frac{\partial}{\partial p_k}\,.
\]
We see that $d_x \Psi(X)=V(q, p)$ when $(q, p)=(q(x), p(x)) \in \Psi(\TT^{n})$, that is, $V$ is tangent to $\Psi(\TT^{n})$ and there, it is equal to $\Psi_{*} X$. 

The vector field $V$ defines the homogeneous quadratic ODE:
\[
\frac{d q_k}{dt}=\sum_{k'} (a_{k'} \cdot k\, q_{k'} +b_{k'} \cdot k\, p_{k'}) p_k
\]
\[
\frac{d p_k}{dt}=- \sum_{k'} (a_{k'} \cdot k\, q_{k'} +b_{k'} \cdot k\, p_{k'}) q_k
\]

To check that this ODE satisfies the antisymmetry condition in Proposition \ref{aux1}, let us relabel the coordinates in $\RR^{2 d(\Lambda)}$ as $x_{\alpha}$, with $\alpha \in \{1,..., 2d(\Lambda)\}$. We see that the coefficients $B_{\alpha \beta \gamma}=0$ unless $x_{\alpha}=q_{k}$ and $x_{\gamma}=p_{k}$ or viceversa. In that case, we have either
\[
B_{q_k, q_{k'}, p_{k}}= a_{k'}\cdot k\, q_{k'}=-B_{p_k, q_{k'}, q_k}
\]
or
\[
B_{q_k, p_{k'}, p_{k}}= b_{k'}\cdot k\, p_{k'}=-B_{p_k, q_{k'}, q_k}
\]
Thus $B_{\alpha \beta \gamma}=-B_{\gamma \beta \alpha}$, that is, the coefficients of the ODE are always antisymmetric under exchange of the first and last index, as we wanted to show.

It remains to be shown that, if $X$ is divergence-free, $V$ is also divergence-free, i.e
\[
\Div V=\sum_k \frac{\partial {V_{q_k}}}{\partial q_k}+\frac{\partial {V_{p_k}}}{\partial p_k}=0
\]

Indeed, we have
\[
\frac{\partial {V_{q_k}}}{\partial q_k}=a_{k} \cdot k\, p_k
\]
\[
\frac{\partial {V_{p_k}}}{\partial p_k}=b_{k} \cdot k\, q_k
\]

and if $X$ is divergence-free, the coefficients in its Fourier expansion satisfy $a_k \cdot k=b_k \cdot k=0$.

\subsection{Proof for $\MM=\SS^{n}$}
Let $\{A_{\mu}\}_{\mu=1}^m$ be a basis of the Lie algebra $\mathfrak{so}(n+1)$ of $(n+1) \times (n+1)$ traceless antisymmetric matrices (so here $m=\frac{n(n+1)}{2}$). It is easy to check that the vector fields in $\RR^{n+1}$ given by
\begin{equation}\label{hopf}
h_{\mu}(x)=A_{\mu} \cdot x
\end{equation}
(where $A_{\mu} \cdot x$ denotes matrix multiplication of $A_{\mu}$ and the vector $x\in \RR^{n+1}$) are tangent to $\SS^{n} \subset \RR^{n+1}$ and, moreover, for any $x \in \SS^{n}$,
\[
\text{ span }\{h_1(x),...,h_m(x)\}=T_{x} \SS^{n}
\]
The vector field $X$ in $\SS^{n}$ can thus be written as
\[
X=\sum_{\mu=1}^{m} f_{\mu} h_{\mu}
\] 
where $f_{\mu} : \SS^{n} \rightarrow \RR$ are smooth functions. 

For any given $N\in \NN$, let $\{Y_{\alpha}\}_{\alpha=1}^{d(N)}$ be a $L^2$-orthonormal basis of the space of spherical harmonics in $\SS^{n}$ of degree up to $N$. We recall that this is the linear space spanned by the eigenfunctions of the Laplace-Beltrami operator (defined with the round metric) of eigenvalues up to $N(N+n-1)$. Equivalently, they are obtained as the restrictions to $\SS^{n}$ of the homogeneous harmonic polynomials in $\RR^{n+1}$ of degrees up to $N$. We label the elements of the basis in increasing degree, so that if the degree of $Y_{\alpha}$ is less than the degree of $Y_{\beta}$, $\alpha < \beta$. 

Note that the components of the vector fields $\{h_{\mu}\}$ are homogeneous polynomials of degree $1$. Since the vector field $X$ is the restriction to $\SS^{n}$ of a polynomial vector field in $\RR^{n+1}$, its components  $f_{\mu}$ must be given by a finite sum of spherical harmonics
\[
f_{\mu}= \sum_{\alpha=1}^{d(N)} c^{\alpha}_{\mu} Y_{\alpha}
\]
up to some degree $N$. 

Define a map $\Psi: \SS^{n} \rightarrow \RR^{d(N)}$ as
\[
\Psi(x)=(Y_{1}(x),...,Y_{d(N)}(x)) \,.
\]

As in the case of the torus, our first goal is to prove that the map $\Psi$ is an embedding. It suffices to do so for $N=1$, where the spherical harmonics are just restrictions to $\SS^n$ of affine functions on $\RR^{n+1}$, so the map is simply
\[
\Psi(x)=(1, x_1(x), x_2(x), ....,x_{n+1}(x))
\]
or some coordinate permutation and rotation of the above, depending on our choice of basis of spherical harmonics. This is clearly an embedding. 
\begin{remark}\label{harmonics2}
As in the case of the torus, if the functions $f_{\mu}$ have zero mean (i.e, they do not have projection into the constant functions on $\SS^n$), we can define the map $\Psi$ as 
\[
\Psi(x)=(Y_2(x), Y_3(x),...,Y_{d(N)}(x))\,.
\]
\end{remark}

We have
\[
d_x \Psi(X)=\sum_{\mu=1}^{m} \sum_{\alpha=1}^{d(N)} c^{\alpha}_{\mu} Y_{\alpha}(x)  ( h_{\mu}(Y_{1})(x),...,h_{\mu}(Y_{d(N)}(x)))\,.
\]

Here $h_{\mu}(Y_{\beta})(x)$ represents the derivative of the eigenfunction $Y_{\beta}$ in the direction of the vector field $h_{\mu}$. We now claim that we can express these derivatives as linear combinations of the original spherical harmonics $\{Y_{\alpha}\}$, with $\alpha \leq d(N)$. In other words, there are explicit coefficients $\theta^{\gamma}_{\mu \beta}$, with $\mu=1,...,m$, $\gamma, \beta=1,...,d(N)$ so that
\[
h_{\mu}(Y_{\beta})(x)=\sum_{\gamma=1}^{d(N)} \theta^{\gamma}_{\mu \beta} Y_{\gamma}(x)\,.
\]
These coefficients will depend on the basis of harmonics $\{Y_{\alpha}\}$, but not on the point $x$ (and actually, they will be zero unless the degree of $Y_{\gamma}$ matches that of $Y_{\beta}$, but we will not use this property).

The existence of such coefficients is a straightforward consequence of the relationship between the spherical harmonics  in $\SS^n$ and the representations of the group $SO(n+1)$, but in order to keep the article as elementary and self-contained as possible, we will give here a simple proof. 

Recall that the vector fields $h_{\mu}$ are given by
\[
h_{\mu}(x)=A_{\mu} \cdot x\,.
\]
For any $t \in \RR$, the matrices $\Lambda^{t}_{\mu}:=\exp (t A_{\mu})$ are elements of $SO(n+1)$. Thus, if $P(x)$ is an harmonic polynomial of degree $\alpha$, $Q(x):=P(\Lambda^{t}_{\mu} x)$ is also an harmonic polynomial of the same degree. 

Hence the spherical harmonics, being restrictions to $\SS^n$ of the harmonic polynomials, inherit this invariance under the action of $\Lambda^{t}_{\mu}$, so that for any basis element $Y_{\beta}$ we have
\[
Y_{\beta}(\Lambda^{t}_{\mu} x)=\sum_{\gamma=1}^{d(N)} \Theta^{\gamma}_{\beta}(\Lambda^{t}_{\mu}) Y_{\gamma}(x)
\]
where the coefficients $\Theta^{\gamma}_{\beta}(\Lambda^{t}_{\mu})$ do not depend on the point.

Now observe that
\[
\frac{d}{dt}\bigg|_{t=0} Y_{\beta}(\exp(t A_{\mu}) x)= h_{\mu}(Y_{\gamma})(x)\,,
\]
so defining 
\[
\theta^{\gamma}_{\mu \beta}=\frac{d}{dt}\bigg|_{t=0} \Theta^{\gamma}_{\beta}(\exp(t A_\mu))
\]
our claim follows.

Hence 
\[
d_x \Psi(X)=\sum_{\mu=1}^{m} \sum_{\alpha=1}^{d(N)} \sum_{\gamma=1}^{d(N)} c^{\alpha}_{\mu}  Y_{\gamma}(x) Y_{\alpha}(x)  (\theta^{\gamma}_{\mu 1},...,\theta^{\gamma}_{\mu d(N)} )\,.
\]

This means that the quadratic vector field in $\RR^{d(N)}$ defined as
\[
V=\sum_{i=1}^{3} \sum_{\alpha, \beta, \gamma =1}^{d(N)} \theta^{\gamma}_{i \beta} c^{\alpha}_{\mu} y_{\gamma} y_{\alpha} \frac{\partial}{\partial y_{\beta}}
\]
is tangent to $\Psi(\SS^{n})$, and coincides there with $d \Psi(X)$.

Furthermore, the coefficients $\theta^{\gamma}_{\mu\beta}$ are antisymmetric in $\gamma, \beta$.:
\[
\theta^{\gamma}_{\mu \beta}=\int_{\SS^{n}} h_{\mu}(Y_{\beta})(x) Y_{\gamma}(x) d \Omega(x)=-\int_{\SS^{n}} h_{\mu}(Y_{\gamma})(x) Y_{\beta}(x) d \Omega(x)=-\theta^{\beta}_{\mu \gamma}
\]
(here we have integrated by parts and used the fact that the vector fields $h_{\mu}$, being infinitesimal generators of isometries, are divergence-free); so the coefficients
\[
B_{\beta \alpha \gamma}:= \sum_{\mu=1}^{m} \theta^{\gamma}_{\mu\beta} c^{\alpha}_{\mu}
\]
satisfy $B_{\beta \alpha \gamma}=-B_{\gamma \alpha \beta}$, as we wanted to show.

Finally, we are left to prove that, assuming the vector field $X$ in $\SS^{n}$ is divergence-free with respect to the round metric, $V$ in $\RR^{d(N)}$ is divergence-free with respect to the Euclidean metric. 

In terms of the expansion in spherical harmonics, the divergence-free condition reads
\[
\Div_{\SS^{n}} X=\sum_{\mu=1}^{m} \sum_{\alpha=1}^{d(N)} c^{\alpha}_{\mu} h_{\mu}(Y_{\alpha})=\sum_{\mu=1}^{m} \sum_{\alpha, \beta=1}^{d(N)} c^{\alpha}_{\mu} \theta^{\alpha}_{\mu \beta} Y_{\beta}=0\,.
\]
Thus we conclude that, if $X$ is divergence-free, the coefficients $c^{\alpha}_{\mu}$ satisfy, for any $\beta$:
\begin{equation}\label{div}
\sum_{\mu=1}^{m} \sum_{\alpha=1}^{d(N)} c^{\alpha}_{\mu} \theta^{\alpha}_{\mu \beta}=0\,,
\end{equation}

so
\[
\Div_{\RR^{d(N)}} V= \sum_{\mu=1}^{m} \sum_{\alpha, \beta, \gamma =1}^{d(N)} \theta^{\gamma}_{\mu \beta} c^{\beta}_{\mu} y_{\gamma} +\theta^{\beta}_{\mu \beta} c^{\alpha}_{\mu} y_{\alpha}=0
\]
(here we have also used the fact that $\theta^{\beta}_{\mu \beta}=0$ because of antisymmetry).

\section{Proof of Theorem \ref{main}}\label{s2}

The idea of the proof is as follows: first the manifold $N$ is embedded into a sphere $\SS^n$ of suitable dimension, and the push-forward of the vector field $X$ is extended to the whole $\SS^n$. By constructing this extension in a suitable way, we can ensure that any other vector field close enough to the extended vector field has an invariant manifold diffeomorphic to $N$, on which it is very close to $X$. We then take a polynomic approximation of the extension of $X$ and apply Theorem \ref{main2} to conclude. 

More precisely, let $F: N \rightarrow \SS^{n}$ be an embedding of $N$ into $\SS^n$. Provided we take the dimension of the sphere high enough, such an embedding always exists.

Our aim now is to extend the vector field $F_{*}(X)$ to the whole $\SS^{n}$, so that $F(N)$ is an $r$-normally hyperbolic invariant manifold:
\begin{definition}[Normally hyperbolic invariant manifold, \cite{HSS} Section 1]
Let $Y$ be a smooth vector field on a manifold $M$. We will say that a submanifold $V \subset M$ is an $r$-normally hyperbolic invariant manifold of $Y$ if $Y$ is tangent to $V$ and, moreover, there is a continuous splitting
\[
TM|_{V}=E^{u} \oplus TV \oplus E^{s} 
\]
and constants $c>0$, $0\leq \mu< \lambda$ such that, for any $x \in V$ we have:
\begin{enumerate}

\item For any $t\in \RR$,  $d_x\phi_{Y}^{t}(E^{u}_{x})=E^{u}_{\phi^{t}_{Y}(x)} $, and analogously for $E^{s}$.

\item For any $v \in E^{s}_{x}$ and $t\geq 0$, $\|d_x\phi_{Y}^{t}(v)\| \leq c e^{-\lambda t}\|v\|$.

\item For any $v \in E^{u}_{x}$ and $t\geq 0$, $\|d_x\phi_{Y}^{-t}(v)\| \leq c e^{-\lambda t}\|v\|$.

\item For any $v \in T_xV$ and $t\in \RR$, $\|d_x\phi_{Y}^{t}(v)\| \leq c e^{\frac{\mu}{r}|t|}\|v\|$.

\end{enumerate}

\end{definition}

Let $(x, z_1,...,z_k)$, $x \in N$, $(z_1,...,z_k) \in N_{F(x)} F(N)$, $k=n-\text{ dim }N$, be local coordinates parametrizing a tubular neighbourhood $V$ of $F(N)\subset \SS^n$. Define a vector field $Z$ that is given in the local coordinates by
\[
Z(x, z_1,...,z_k)=X(x)-C(z_1 \frac{\partial}{\partial z_1}+...+z_k \frac{\partial}{\partial z_k})\,.
\]
Because $N$ is compact, if we take the constant $C$ large enough the field $Z$ is $r$-normally hyperbolic on $F(N)$ for any a priori chosen $r$. We then extend it smoothly in an arbitrary way to the rest of $\SS^n$, still denoting this extension by $Z$. 

A crucial property of $r$-normally hyperbolic flows is their structural stability (Theorem 4.1 in \cite{HSS}), that is, any other vector field $Z'$ close enough to $Z$ in the $C^r$ norm, $r\geq 1$, has the following property: there is an embedding $F': N \rightarrow \SS^n$, close to $F$ in the $C^r$ norm, such that $Z'$ is tangent to $F'(N)$. This also implies, in particular, that the vector field $dF'^{-1}( Z'|_{F'(N)})$ on $N$ is close to $X$ in the $C^{r-1}$ norm. 

We now take $Z'$ to be a polynomic vector field on $\SS^n$, approximating $Z$ to any desired degree of accuracy (more precisely, we approximate by polynomials the components of the vector field $Z$ with respect to the vector fields $h_{\mu}$ defined in Section \ref{s11}). Theorem \ref{main} follows by applying Theorem \ref{main2} to $Z'$. 

\subsection{The dimension of $M$}\label{dimension}
The dimension of the Riemannian manifold $M$ in Theorem \ref{main} will depend only on the dimension $n$ of the sphere in which we embed $N$, and on the degree of the polynomic vector field $Z'$ approximating the extension of $X$. More precisely, $M=SO(d)\times \TT^d$, where $d$ is the dimension of the space of spherical harmonics in $\SS^n$ whose degrees are less than that of $Z'$ (see Subsection \ref{dimension2}). The degree of $Z'$ will in turn depend on the acceptable error $\epsilon$ in the approximation of $Z$ by $Z'$, which in applications will be determined by the robustness of the dynamical feature of $X$ we are interested in. 

To get a quantitative bound on the degree, and hence on the dimension, as a function of $\epsilon$, we can use a multidimensional Jackson-type theorem. For example, by virtue of Theorem 2 in \cite{Bagby}, we have that, for any integer $k\geq 0$, there are  polynomials $Z'$ of any degree satisfying:
\[
||Z-Z'||_{C^{m}(V)} \leq \frac{c}{(\text{ degree }Z')^{k}} ||Z||_{C^{m+k}(V)}
\]
where $c$ is a constant depending on the dimension $n$, on the desired norm of approximation $m$, and on the tubular neighbourhood $V$ of $F(N)$. Observe that, by how the extension $Z$ was constructed, the $C^{m+k}$ norms of $Z$ can be bounded by those of $X$ in $N$, modulo a constant depending on the Lyapunov exponents of $X$ and on the derivatives of the embedding of $N$.

Thus we conclude
\[
\text{ degree } Z'\leq \bigg(\frac{C'}{\epsilon}||X||_{C^{m+k}(N)}\bigg)^{\frac{1}{k}}
\]
so that the parameter $d$ in $SO(d)\times \TT^d$ can be bounded by the dimension of the space of harmonic polynomials of degrees up to the degree $D$ of $Z'$
\[
d < \binom{D+n}{n}\frac{D(2Dn+n^2+1)}{n(n+1)(D+n)} \sim D^{n+1}
\]
and finally
\[
\text{ dim } M < \bigg(\frac{C'}{\epsilon} ||X||_{C^{m+k}(N)}\bigg)^{2\frac{n+1}{k}}\,.
\]

\section{Proof of Theorem \ref{chaos}}\label{s4}

The family of vector fields on $\TT^3$ defined as
\[
u_{ABC}(x_1, x_2, x_3) = (A \sin x_3 + C \cos x_2) \frac{\partial}{\partial x_1} + (B \sin x_1 + A \cos x_3) \frac{\partial}{\partial x_2} + 
\]
\[
+(C \sin x_2 + B \cos x_1) \frac{\partial}{\partial x_3}
\]
for parameters $A, B, C \in \RR$, are called the ABC (Arnold-Beltrami-Childress) flows. They have been extensively studied (see e.g \cite{AKh} and references therein), and are known to have chaotic invariant sets for certain values of the parameters $A, B, C$ \cite{Ch, Zi}. 

An ABC flow is given by sines and cosines with integer frequencies in the unit sphere. Thus, arguing as in the proof of Theorem \ref{main2} for the torus, we see that ABC vector fields are embedded in the Euler dynamics on $M=SO(6) \times \TT^{6}$ (there are seven integer points in the ball of radius $1$, but we can exclude the zero vector by virtue of Remark \ref{harmonics}). Theorem \ref{chaos} follows.

\section{Additional comments}\label{sf}

Here we give some additional constructions of interesting dynamics inside the Euler system which do not follow immediately from Theorems \ref{main} and \ref{main2}, but rather need the concert of other results:

\subsection{The Lorenz attractor in the Euler equations}

The Lorenz attractor is a paradigmatic example of attractor and a popular emblem of chaos. It arises in an ODE in $\RR^3$ that is obtained from a Galerkin truncation of the Boussinesq equation (itself a PDE approximating the Navier-Stokes equations).  

The goal of this subsection is to embed 3-dimensional geometric Lorenz flows (vector fields introduced in \cite{W} that have the same qualitative dynamics as the Lorenz system) into the Euler dynamical system on the manifold $SO(d)\times \TT^{d}$, so that the Euler equations reduce to the Lorenz dynamics in a finite dimensional subset of the phase space. In other words, in that manifold, the Lorenz dynamics are not a toy model of the Navier-Stokes equations, but an exact description of ideal fluid motion. 

Lorenz attractors are not stable under perturbation of the flow \cite{W}, which precludes the direct application of Theorem \ref{main}. 

Nevertheless, by a theorem of Guckenheimer and Williams (see the main Theorem in \cite{GW}), the set of vector fields in $\RR^3$ having a geometric Lorenz attractor contains an open set in the $C^{0}$ topology. Since geometric Lorenz attractors are contained in a bounded set, this fact carries over to any other $3$-dimensional compact manifold; in particular, to $\SS^3$. Thus there are polynomial vector fields in $\SS^3$ having a geometric Lorenz attractor, because they are dense for the $C^{0}$ topology. Applying Theorem \ref{main2} to one of these vector fields, we obtain a finite-dimensional family of solutions to the Euler equation on some Riemannian manifold $M$ that converge to a geometric Lorenz attractor. 

\subsection{The universal template}

There is a Riemannian manifold $M$, and a 3-dimensional family of divergence-free vector fields $\Sigma \subset \vect^{\infty}(M)$, with the following properties:
\begin{enumerate}
\item $\Sigma$ is diffeomorphic to a 3-sphere, and invariant under the Euler dynamical system.
\item The Euler dynamics on $\Sigma$ contains sets of periodic orbits representing every isotopy class of knots. 
\end{enumerate}

Indeed, there is a vector field in $\SS^{3}$ containing periodic orbits of every isotopy class of knots and links, which are moreover stable under sufficiently small $C^{m}$ perturbations (Corollary 3.2.19 in \cite{GS}). Applying Theorem \ref{main}, the claim follows.

\subsection{Euler trajectories and translation surfaces of triangular billiards}\label{sbilliards}

Let $P\subset \RR^2$ be a triangle. We will henceforth assume that $P$ comprises the interior and the sides of the triangle, but not the vertices. 

The billiard on $P$ is the dynamical system defined thus: a particle at point $q \in P$, with initial velocity $v$ of modulus $1$, moves with constant velocity while in the interior of $P$, and is reflected every time the trajectory hits a side of the triangle. If the trajectory hits a vertex, the particle stops. Any billiard trajectory is thus completely determined by the initial position $q \in P$ and the angle $\theta \in  \RR / 2 \pi \ZZ$ of the initial velocity.

By means of the unfolding construction of Katok-Zemljakov  \cite{KZ}, we can associate to any triangle $P$ an open, smooth Riemann surface $S_{P}$ with a flat metric, so that the geodesic flow on the unit tangent bundle of $S_{P}$ is equivalent to the billiard on $P$.  

The surface $S_{P}$ can be endowed with an atlas whose transition maps are euclidean translations, and so it is called the translation surface of $P$. In the coordinates given by this atlas, the geodesics on $S_{P}$ are straight lines, and their slope is globally well-defined, because the transition functions are translations. Thus for any given angle $\theta$, we can define a foliation on $S_{P}$ whose leaves are the geodesics of slope $\tan(\theta)$.

\begin{corollary}\label{billiards}

For any triangle $P$, there is a metric on $M=SO(30)\times \TT^{30}$ so that the Euler equation on $M$ has the following property: for any angle $\theta$, there is an compact surface $\Sigma_{P, \theta} \subset \vect^{\infty}(M)$, invariant under the Euler evolution, and such that
\begin{enumerate}
\item The surface $\Sigma_{P, \theta}$ minus a finite number of points is diffeomorphic to the translation surface $S_{P}$. These finite points in $\Sigma_{P, \theta}$ are stationary solutions of the Euler equation.
\item The one-dimensional foliation on $\Sigma_{P, \theta}$ whose leaves are given by the Euler trajectories is diffeomorphic to the foliation on $S_{P}$ given by geodesics of slope $\tan(\theta)$. 

\end{enumerate}

\end{corollary}

\begin{proof}                                                                                                                                                                                                                                                                                                                                                            
They key ingredient in the proof is the dictionary between polygonal billiards and homogeneous foliations in $\CC^2$ due to F. Valdez \cite{V}. Let $\lambda_{1}, \lambda_2, \lambda_3$ be the angles of the triangle $P$. In $\CC^2$ we define the following homogeneous holomorphic vector field:
\[
X=z_1(\lambda_3 z_2+\lambda_2(z_2-z_1)) \partial_{z_1}+z_2(\lambda_3 z_1+\lambda_1(z_1-z_2))
\]
The integral curves of $\Re (X)$ and $\Im (X)$ generate a $2$-dimensional homogeneous foliation $\mathcal{F}$ in $\RR^4$.

By virtue of Theorem 1.1 in  \cite{V}, for any angle $\theta$ we can find a leaf $L$ of $\mathcal{F}$ that is diffeomorphic to $S_P$ through a diffeomorphism that maps the integral curves of $\Re(X)$ in $L$ to the leaves of the foliation in $S_{P}$ given by the geodesics of slope $\tan(\theta)$. 

As argued in Section 1.1 of \cite{V}, the homogeneity of the construction above allows us to deduce the analogous result in $\RP^{3}$ and $\SS^3$. In particular, the foliation in $\SS^3$ defined by the integral curves of the vector fields
\[
U=\Re(X)-(\Re(X)\cdot \partial_{r}) \partial_{r}
\]
\[
V=\Im(X)-(\Im(X)\cdot \partial_{r}) \partial_{r}
\]
(where $\partial_r$ denotes the radial unit vector field) has a leaf which, minus the zeros of $U$, is diffeomorphic to $S_P$, and on which $U$ defines a foliation diffeomorphic to the geodesic foliation of any given slope.

The vector fields $U$ and $V$ in $\SS^3$ are the restrictions of polynomial vector fields of degree $4$ on $\RR^4$. Arguing as in the proof of Theorem \ref{main2} in the case of the $\SS^3$, we see that we can embed the vector field in $SO(30)\times \TT^{30}$ ($30$ being the dimension of the linear space of spherical harmonics in $\SS^3$ of degree up to $3$). Corollary \ref{billiards} follows. 

\end{proof}

\begin{remark}
Corollary \ref{billiards} can be generalized to more general polygonal billiards using the construction in Section 5 of \cite{V}. 
\end{remark}

\section{Acknowledgements}

This work owes a great deal to Daniel Peralta-Salas, who shared with the author many crucial insights. The author also wants to thank Theodore Drivas, Boris Khesin and \'Angel David Mart\'inez for useful conversations and for helping to improve the manuscript with their suggestions. Finally, we acknowledge the excellent working conditions and financial support provided by the Max Planck Institute for Mathematics, the University of Toronto, and the Fields Institute.


\begin{thebibliography}{99}\frenchspacing


\bibitem{AKh}
V.I. Arnold, B.A. Khesin,
\emph{Topological Methods in Hydrodynamics}. Springer-Verlag,
New York 1998

\bibitem{Bagby}
T. Bagby, L. Bos, N. Levenberg, Multivariate simultaneous approximation. Constr. Approx. 18 (2002), 4 569–-577

\bibitem{CMDP}

R. Cardona, E. Miranda, D. Peralta-Salas, F. Presas, Universality of Euler flows and flexibility of Reeb embeddings, arXiv:1911.01963

\bibitem{Ch}

C. Chicone, A geometric approach to regular perturbation theory with an application to hydrodynamics. 
Trans. Amer. Math. Soc. 347 (1995) 12 4559--4598


\bibitem{CF}
N. Crouseilles, E. Faou,
Quasi-periodic solutions of the 2D Euler equation. Asymptot. Anal. 81 (2013) 1 31--34




\bibitem{EB}
D.G. Ebin, J.E. Marsden, Groups of diffeomorphisms and the motion of an incompressible fluid. Ann. of Math. 92 (1970) 102--163



\bibitem{GS}
R. Ghrist, P. J. Holmes,  M. C. Sullivan, \emph{Knots and links in three-dimensional flows}, Lecture Notes in Mathematics 1654,  Springer-Verlag, Berlin 1997

\bibitem{GW}
J. Guckenheimer, R. F. Williams, Structural stability of Lorenz attractors, Inst. Hautes Études Sci. Publ. Math. 50 (1979), 59–-72



\bibitem{HSS}

M. W. Hirsch, C. C. Pugh, M. Shub, \emph{Invariant manifolds}, Lecture Notes in Mathematics 583, Springer-Verlag, Berlin-New York 1977

\bibitem{KZ}

A. N. Zemljakov, A. B. Katok, Topological transitivity of billiards in polygons. Math. Notes 18 (1975) 1--2 760–-764

\bibitem{KKPS1}

B. Khesin, S. Kuksin, D. Peralta-Salas, KAM theory and the 3D Euler equation. Adv. Math. 267 (2014) 498--522

\bibitem{KKPS2}

B. Khesin, S. Kuksin, D. Peralta-Salas, Global, local and dense non-mixing of the 3D Euler equation, Arch. Ration. Mech. Anal. 238 (2020) 3 1087--1112


\bibitem{Na}

N. Nadirashvili, Wandering solutions of the two-dimensional Euler equation,
Funct. Anal. Appl. 25 (1991) 3 220--221





\bibitem{Sn}

A. Shnirelman, Evolution of singularities, generalized Liapunov function
and generalized integral for an ideal incompressible fluid, Amer. J. Math.
119 (1997) 579--608



\bibitem{Taoquadratic}

T. Tao, On the universality of the incompressible Euler equation on compact manifolds, Discrete Contin. Dyn. Syst. 38 (2018) 3 1553--1565.


\bibitem{V}

F. Valdez, Billiards in polygons and homogeneous foliations on $\CC^2$, Ergodic Theory Dynam. Systems 29 (2009) 255-–271

\bibitem{W}
R. F. Williams, The structure of Lorenz attractors.
Inst. Hautes Études Sci. Publ. Math. 50 (1979) 73–-99

\bibitem{Wo}

W. Wolibner, Un theor\`eme sur l'existence du mouvement plan d'un fluide parfait, homog\`ene, incompressible, pendant un temps infiniment long. Math. Z. 37 (1933) 1 698--726





\bibitem{Zi}
S. L. Ziglin, Splitting of the separatrices and the nonexistence of first integrals in systems of differential equations of Hamiltonian type with two degrees of freedom, Math. USSR Izv. 31 (1988) 407--421


\end{thebibliography}
\end{document}